\documentclass[11pt,a4paper]{article}

\usepackage[T1]{fontenc}
\usepackage{lmodern}
\usepackage[margin=2.6cm]{geometry}
\usepackage{amsmath,amssymb,amsthm,mathtools}
\usepackage{booktabs,array}
\usepackage{enumitem}
\usepackage{needspace}
\usepackage{xcolor}
\usepackage{microtype}
\usepackage[colorlinks=true,linkcolor=blue!60!black,citecolor=blue!60!black,urlcolor=blue!60!black]{hyperref}

\setlist{nosep}
\allowdisplaybreaks

\newtheorem{theorem}{Theorem}[section]
\newtheorem{proposition}[theorem]{Proposition}
\newtheorem{lemma}[theorem]{Lemma}

\theoremstyle{definition}

\theoremstyle{remark}

\newcommand{\C}{\mathbb C}
\newcommand{\F}{\mathbb F}
\newcommand{\Q}{\mathbb Q}
\newcommand{\Z}{\mathbb Z}

\newcommand{\calB}{\mathcal B}

\newcommand{\Dim}{\operatorname{Dim}}

\title{A Multiplicative Horn Obstruction to the Unitary\\
Categorification of a Rank-Twelve Modular Datum}
\author{
  Jiuhe Liu\\
  \small Qiuzhen College, Tsinghua University\\
  \small \href{mailto:liujiuhe21@mails.tsinghua.edu.cn}{liujiuhe21@mails.tsinghua.edu.cn}
  \and
  Xuzhihan Ding\\
  \small Qiuzhen College, Tsinghua University\\
  \small \href{mailto:dingxzh21@mails.tsinghua.edu.cn}{dingxzh21@mails.tsinghua.edu.cn}
}
\date{\today}

\begin{document}
\maketitle

\begin{abstract}
We study the rank-twelve potential modular datum denoted
$12^{39,406}_{4,1276.}$ in the low-rank classification tables.  Its global
dimension is $(1287+351\sqrt{13})/2$, and its fusion ring agrees with that of
the known rank-twelve Haagerup modular category.  The datum satisfies the
modular relations, has nonnegative integral Verlinde coefficients, and passes
several familiar genus-one necessary tests.  We prove,
nevertheless, that it cannot be realized by a unitary modular tensor category.

The obstruction already occurs in the representation of the three-strand
braid group on $\operatorname{Hom}(C_{[10]},B^{\otimes3})$.  Generalized
Frobenius--Schur indicators determine the eigenvalues of the two-strand
generator on this twelve-dimensional space.  The central relation in $B_3$
then reduces the problem to a product of three prescribed unitary conjugacy
classes.  A sixfold symmetry reduces the $1302$ determinant-compatible cases
to $217$ representatives.  Rank-one inequalities leave only $17$ of these,
and ten explicitly displayed degree-zero rank-two Schubert inequalities
exclude the remaining cases.  A violated inequality with gap $2/117$ is
worked out explicitly.

All calculations used in the obstruction are exact.  The argument uses
unitarity essentially and leaves non-unitary categorification open.
\end{abstract}

\tableofcontents

\section{Introduction}

Modular data are among the most accessible invariants of a modular tensor
category.  A pair of matrices $(S,T)$ records the mapping-class-group action
at genus one, and the Verlinde formula recovers the fusion coefficients from
$S$.  The converse problem is substantially harder: matrices satisfying the
modular relations, positivity, integrality, and standard arithmetic
constraints need not arise from a category.  Following common usage, we call
such a pair a \emph{potential modular datum} until a compatible fusion
category, associator, braiding, and ribbon structure have been constructed.

This paper treats one entry in the integral rank-twelve part of the
classification of Ng, Rowell, and Wen \cite{NgRowellWenClassification}.  The
entry is labelled
\[
  12^{39,406}_{4,1276.}
\]
and comes from the metric-group construction of Grossman and Izumi
\cite{GrossmanIzumi}.  It shares its based fusion ring with the modular datum
of the Drinfeld center of the Haagerup fusion category, but its twists are
different.  Direct genus-one tests do not decide whether it is categorifiable.

Our main result gives a negative answer in the unitary setting.

\begin{theorem}[Main theorem]\label{thm:main}
There is no unitary modular tensor category whose modular data are the
potential modular datum $12^{39,406}_{4,1276.}$.
\end{theorem}

The proof does not solve the pentagon equations.  Instead, it uses a necessary
three-strand braid representation.  Let $B$ be the nontrivial twist-one object
of dimension $\delta+1$ and set
\[
 V_t=\operatorname{Hom}(t,B^{\otimes3}).
\]
The generalized indicators of Ng, Rowell, and Wen determine the eigenvalues
of the braid generator on each $V_t$ \cite{NgRowellWenR}.  For
$t=C_{[10]}$, the space has dimension twelve.  After quotienting the central
character of $B_3$, the braid relation asks for unitary matrices $P,Q$ with
$P^3=Q^2=I$ and a prescribed spectrum for $P^{-1}Q$.  The multiplicative Horn
theorem translates the existence of such matrices into quantum Schubert
inequalities \cite{AgnihotriWoodward,Belkale}.  Exact rank-one and rank-two
inequalities already exclude every possible eigenvalue multiplicity and
central branch.

The standard genus-one consistency tests do not detect the obstruction.  The
three-strand representation supplies the additional categorical constraint
needed here.

\subsection*{Organization}
Section~\ref{sec:related} places the result in context and fixes its scope.
Section~\ref{sec:data} describes the candidate and our conventions.
Section~\ref{sec:indicators} gives the braid eigenvalues.  The proof of
Theorem~\ref{thm:main} is in Section~\ref{sec:horn}.
Section~\ref{sec:nonunitary} isolates the remaining non-unitary question.
Appendix~\ref{sec:certificates} describes the exact certificates.

\section{Related work and scope}\label{sec:related}

\subsection{Potential modular data and quadratic categories}
Grossman and Izumi constructed infinite families of potential modular data
from pairs of involutive metric groups whose Gauss sums differ by a sign
\cite{GrossmanIzumi}.  Their construction includes modular data motivated by
centers of generalized Haagerup and near-group categories, but the construction
of $(S,T)$ alone does not produce a fusion category.  The classification of
Ng, Rowell, and Wen determines modular data through rank eleven and integral
modular data through rank twelve \cite{NgRowellWenClassification}.  The word
``classification'' in that setting concerns solutions of the modular-data
constraints; categorification remains a separate problem.

Evans and Gannon computed modular data associated with near-group categories
and their doubles and analyzed the relevant $C^*$-near-group systems
\cite{EvansGannon}.  Their work supplies the principal realized comparison
datum in the present fusion ring, but does not decide whether the candidate
twists are categorifiable.

\subsection{Beyond genus one}
Several works emphasize that modular data do not determine all categorical
structure and that additional realizability constraints are required; see,
for example, \cite{BondersonRowellWang}.  Generalized Frobenius--Schur
indicators provide one bridge from modular data to local braiding data.
Ng, Rowell, and Wen show that, in a suitable gauge, $R$-symbols in
multiplicity-one channels can be recovered from twists, fusion coefficients,
and second generalized indicators \cite{NgRowellWenR}.  We use precisely the
gauge-invariant consequence needed here: the spectrum of a braid generator on
a fusion-path space.

The product-eigenvalue problem for unitary matrices was solved in terms of
quantum Schubert calculus by Agnihotri--Woodward and independently by Belkale
\cite{AgnihotriWoodward,Belkale}.  Our main observation is that a categorical
$B_3$ representation converts the candidate $R$-spectrum into one of these
finite product problems.  Thus a violation of a multiplicative Horn
inequality is a categorical obstruction, even though the inequality itself is
formulated entirely in terms of unitary spectra.

\subsection{Precise scope}
Theorem~\ref{thm:main} is a non-realizability theorem for \emph{unitary}
modular tensor categories and does not assume a particular associator.  It is
not, however, a theorem of absolute
non-categorifiability.  In a non-unitarizable ribbon fusion category, the
finite-order matrices obtained from the braid relation need not be unitary for
a common positive-definite Hermitian form.  The multiplicative Horn polytope
is then no longer a necessary condition.  We leave that non-unitary problem
open.

\section{The candidate datum and conventions}\label{sec:data}

Set
\begin{equation}\label{eq:deltaE}
 \delta=\frac{9+3\sqrt{13}}2,
 \qquad E=\delta+2=\frac{13+3\sqrt{13}}2.
\end{equation}
Then
\begin{equation}\label{eq:relations}
 \delta^2=9(\delta+1),\qquad E\overline E=13.
\end{equation}
We order the labels as
\begin{equation}\label{eq:labels}
 L=(A,B,C_{[10]},C_{[01]},C_{[11]},C_{[12]},D_1,\ldots,D_6),
\end{equation}
where $A=\mathbf1$.  Their dimensions are
\begin{equation}\label{eq:dims}
 (1,\delta+1,E,E,E,E,\delta,\ldots,\delta).
\end{equation}
Consequently,
\begin{equation}\label{eq:globaldim}
 \Dim(\calB_{\mathrm{cand}})
 =1+(\delta+1)^2+4E^2+6\delta^2
 =9E^2=\frac{1287+351\sqrt{13}}2.
\end{equation}

The metric-group input is
\begin{equation}\label{eq:metricinput}
 (\F_3^2,q_{\mathrm{an}}),\quad
 q_{\mathrm{an}}(x)=\zeta_3^{x_1^2+x_2^2},
 \qquad
 (\Z_{13},q'_1),\quad q'_1(\gamma)=\zeta_{13}^{\gamma^2}.
\end{equation}
We use the inverse-monodromy convention for $S$.  The twists are
\begin{equation}\label{eq:twists}
 \theta_A=\theta_B=1,\qquad
 \theta_{C_{[k]}}=q_{\mathrm{an}}(k)^{-1},\qquad
 \theta_{D_\gamma}=q'_1(\gamma).
\end{equation}
All matrix identities below are evaluated in $\Q(\zeta_{39})$ unless otherwise
stated.  The exact checks include the modular relations and all $12^3=1728$
Verlinde coefficients.  These facts establish an abstract modular datum with a
nonnegative integral based ring, not a modular category.

In the Grothendieck ring, the fusion rule needed most often is
\begin{equation}\label{eq:Bsquare}
 [B]^2=[A]+[B]+\sum_{i=1}^4[C_i]+\sum_{j=1}^6[D_j].
\end{equation}
In particular, every channel in $B\otimes B$ has multiplicity one.

\section{Braid eigenvalues from generalized indicators}\label{sec:indicators}

For the direct-monodromy convention, the generalized indicator formula of
Ng--Rowell--Wen reads
\begin{equation}\label{eq:indicator}
 \nu_{2,1}^{\iota(c)}(x)
 =\frac1{\Dim(\calB)}\sum_{k,l}d_k\,
 \overline{S^{\mathrm{dir}}_{c,l}}N_{kl}^{\ x}
 \frac{\theta_k^2}{\theta_l^2}.
\end{equation}
Our stored matrix uses inverse monodromy, so
$\overline{S^{\mathrm{dir}}}=S^{\mathrm{inv}}$.  In a multiplicity-one
channel $N_{xx}^{\ c}=1$, choose a square root of $\theta_c$.  Then
\begin{equation}\label{eq:Rsign}
 R_{x,x}^{c}=\varepsilon_{x,c}\frac{\sqrt{\theta_c}}{\theta_x},
 \qquad
 \varepsilon_{x,c}=
 \frac{\nu_{2,1}^{\iota(c)}(x)}{\sqrt{\theta_c}}\in\{\pm1\}.
\end{equation}
Changing the square-root convention changes the displayed sign and root
together, not the resulting eigenvalue.

\begin{proposition}\label{prop:Rsigns}
In the order~\eqref{eq:labels}, the normalized signs in the $B\otimes B$
channels are
\begin{equation}\label{eq:Rsigns}
 (\varepsilon_{B,c})_{c\in L}
 =(1,1,1,1,-1,-1,1,-1,1,1,-1,-1).
\end{equation}
In particular, $R_{B,B}^{A}=R_{B,B}^{B}=1$.
\end{proposition}

The last two equalities show that commutativity on the prospective
$A\oplus B$ multiplication channels is not obstructed at this order.  The
higher braid relation is more restrictive.  Writing
$\zeta_{78}=e^{2\pi i/78}$, the actual braid eigenvalues are
\begin{equation}\label{eq:Rexponents}
 (R_{B,B}^{c})_{c\in L}
 =\bigl(\zeta_{78}^{r_c}\bigr)_{c\in L},\qquad
 (r_c)=(0,0,26,26,52,52,3,51,27,9,75,69).
\end{equation}

\section{The multiplicative Horn obstruction}\label{sec:horn}

For a simple object $t$, define
\begin{equation}\label{eq:Vt}
 V_t=\operatorname{Hom}(t,B^{\otimes3}).
\end{equation}
The fusion tensor gives
\begin{equation}\label{eq:Vdims}
 (\dim V_t)_{t\in L}=(1,12,12,12,12,12,10,10,10,10,10,10).
\end{equation}
In the fusion-path basis, the eigenvalues of
$A_1=\rho_t(\sigma_1)$ are $R_{B,B}^{c}$ with multiplicities
$N_{cB}^{\ t}$.  Since $\sigma_1$ and $\sigma_2$ are conjugate in $B_3$, the
same is true for $\rho_t(\sigma_2)$.

\begin{lemma}[Central reduction]\label{lem:central}
Let $A_1=\rho(\sigma_1)$ and $A_2=\rho(\sigma_2)$ be a unitary
representation of $B_3$ on $V_t$.  Put
\[
 X=A_1A_2A_1,\qquad Y=A_1A_2.
\]
Then $X^2=Y^3=\theta_tI$.  Given $\alpha^2=\theta_t$ and
$\beta^3=\theta_t$, set
\[
 Q=\alpha^{-1}X,\qquad P=\beta^{-1}Y,\qquad
 \lambda=\alpha\beta^{-1}.
\]
Then
\begin{equation}\label{eq:PQ}
 Q^2=P^3=I,\qquad A_1=\lambda P^{-1}Q,\qquad
 \lambda^6=\theta_t.
\end{equation}
Conversely, unitary $P,Q$ satisfying~\eqref{eq:PQ} reconstruct a unitary
$B_3$ representation with the indicated central character.
\end{lemma}

\begin{proof}
The braid relation gives
$(\sigma_1\sigma_2\sigma_1)^2=(\sigma_1\sigma_2)^3$.  In a ribbon
category this central full twist acts on
$\operatorname{Hom}(t,B^{\otimes3})$ as
$\theta_t\theta_B^{-3}=\theta_t$, because $\theta_B=1$.  Moreover
$Y^{-1}X=A_1$, proving the forward implication.  Conversely set
$X=\alpha Q$, $Y=\beta P$, and define
$A_1=Y^{-1}X$ and $A_2=X^{-1}Y^2$.  Direct substitution gives the braid
relation and the claimed central character.
\end{proof}

We now take $t=C_{[10]}$.  The twelve eigenvalue exponents of $A_1$ are
\begin{equation}\label{eq:C10spectrum}
 (0,26,26,26,52,52,3,51,27,9,75,69)\pmod{78},
\end{equation}
and $\theta_t=\zeta_{78}^{52}$.  There are six choices of $\lambda$ in
Lemma~\ref{lem:central}.  For each choice let $D=A_1/\lambda$.  A necessary
factorization is
\begin{equation}\label{eq:triple}
 P^{-1}QD^{-1}=I.
\end{equation}
The eigenvalues of $P$ lie in $\{1,\omega,\omega^2\}$ and those of $Q$ lie in
$\{1,-1\}$.  Write $(m_0,m_1,m_2)$ for the corresponding multiplicities of
$P$, write $\nu$ for the multiplicity of $-1$ for $Q$, and enumerate the six
roots as $\lambda_b=\zeta_6^b\lambda_0$, $b\in\Z/6\Z$.  The determinant
identity in~\eqref{eq:triple} leaves 1302 tuples $(b,\nu;m_0,m_1,m_2)$, but the
following symmetry removes the apparent factor of six.

\begin{lemma}[Branch symmetry]\label{lem:branchsymmetry}
The transformation
\begin{equation}\label{eq:branchaction}
 (b,\nu;m_0,m_1,m_2)
 \longmapsto (b+1,12-\nu;m_1,m_2,m_0)
\end{equation}
preserves solvability of~\eqref{eq:triple}.  It defines a free $\Z/6\Z$-action
on the determinant-compatible tuples.  Consequently it is enough to test the
217 tuples with $b=0$.
\end{lemma}

\begin{proof}
Since $D_{b+1}=\zeta_6^{-1}D_b$, a solution $P^{-1}Q=D_b$ gives a solution in
the next branch by setting
\[
 P'=\omega^2P,\qquad Q'=-Q.
\]
Indeed, $(P')^3=(Q')^2=I$ and
\[
 (P')^{-1}Q'=\omega^{-2}(-1)P^{-1}Q
 =\zeta_6^{-1}D_b=D_{b+1}.
\]
Multiplication by $\omega^2$ cyclically changes the $P$-multiplicities to
$(m_1,m_2,m_0)$, while multiplication of $Q$ by $-1$ changes $\nu$ to
$12-\nu$.
The branch coordinate increases by one, so every orbit has exactly six
elements.
\end{proof}

Multiplying each factor in~\eqref{eq:triple} by a scalar, we normalize it to
$SU(12)$.  Choosing a different twelfth root changes the triple by the standard
central action on the multiplicative Horn polytope, so it suffices to test one
representative of each central orbit.  Write the ordered alcove coordinates of
the three normalized factors as $u,v,w$.  The multiplicative Horn theorem
states that a triple with product $I$ exists if and only if, for every relevant
nonzero three-point genus-zero Gromov--Witten invariant of $G(r,12)$,
\begin{equation}\label{eq:AW}
 \sum_{i\in I}u_i+\sum_{j\in J}v_j+\sum_{k\in K}w_k\le d.
\end{equation}
Here our indexing convention is that the quantum product
$[\sigma_I]\star[\sigma_J]$ contains $q^d[\sigma_{*K}]$; the checker uses the
same Poincare-duality convention when converting partitions to subsets.

\Needspace{17\baselineskip}
We now fix the alcove representatives explicitly.  For a
rational vector $x=(x_1,\ldots,x_{12})$ and a scalar $s$, let
$p_1\ge\cdots\ge p_{12}$ be the decreasing rearrangement of the fractional
parts $\{x_i+s\}$.  Whenever
\[
 c=\sum_{i=1}^{12}p_i\in\Z,
\]
define
\begin{equation}\label{eq:alcovedef}
 \operatorname{Alc}(x;s)=\operatorname{sort}_{\downarrow}
 (p_1-1,\ldots,p_c-1,p_{c+1},\ldots,p_{12}).
\end{equation}
This is the unique decreasing $SU(12)$ alcove lift with zero sum and spread at
most one.  Superscripts on entries below denote repetition.  For $b=0$, put
\begin{align}
 U(m_0,m_1,m_2)&=(0^{m_0},(2/3)^{m_1},(1/3)^{m_2}),\label{eq:Uraw}\\
 V(\nu)&=((1/2)^{\nu},0^{12-\nu}),\label{eq:Vraw}\\
 W_0&=(1/9,(7/9)^3,(4/9)^2,17/234,107/234,179/234,
       233/234,35/234,53/234).\label{eq:Wraw}
\end{align}
Here $U,V,W_0$ are the raw eigenphases of $P^{-1},Q,D_0^{-1}$,
respectively; in particular the chosen phase of $\lambda_0$ is $1/9$.  Set
\begin{equation}\label{eq:scalars}
 s_u=-\frac{2m_1+m_2}{36},\qquad
 s_v=-\frac{\nu}{24},\qquad
 s_w=-s_u-s_v.
\end{equation}
For every determinant-compatible tuple with $b=0$, the coordinates used in
Equation~\eqref{eq:AW} are exactly
\begin{equation}\label{eq:explicituvw}
 u=\operatorname{Alc}(U(m_0,m_1,m_2);s_u),\qquad
 v=\operatorname{Alc}(V(\nu);s_v),\qquad
 w=\operatorname{Alc}(W_0;s_w).
\end{equation}

\begin{theorem}\label{thm:horn}
The spectrum~\eqref{eq:C10spectrum} is incompatible with a unitary
representation of $B_3$ having the ribbon central character
$\theta_{C_{[10]}}$.
\end{theorem}

\begin{proof}
By Lemma~\ref{lem:branchsymmetry}, fix $b=0$.  Evaluation of the rank-one
inequalities in exact rational arithmetic leaves the following 17 pairs
$(\nu;(m_0,m_1,m_2))$:
\[
\begin{array}{c|l}
\nu & (m_0,m_1,m_2)\\ \hline
6 &(3,6,3),(4,4,4),(5,2,5),(5,5,2),(6,3,3),(7,4,1)\\
8 &(1,7,4),(2,5,5),(3,3,6),(3,6,3),(4,1,7),(4,4,4),
    (5,2,5),(5,5,2),(6,3,3),(7,4,1)\\
10&(4,4,4).
\end{array}
\]

Only the ten rank-two inequalities in Table~\ref{tab:compacthorn} are needed
to exclude these cases.  In the last column, $a$ denotes the row number of
the inequality and the displayed fraction is its positive gap
$\textup{LHS}-d$ in~\eqref{eq:AW}.  Table~\ref{tab:compactcoverage} records a
witness for every surviving multiplicity tuple.  Thus the tables constitute
a compact exact cover, rather than an appeal to all 10,296 rank-two terms.

\begin{table}[ht]
\centering
\small
\caption{The ten rank-two Horn inequalities used in the compact certificate.
Every listed Littlewood--Richardson coefficient is one and every degree is
$d=0$.}
\label{tab:compacthorn}
\begin{tabular}{c@{\quad}c@{\quad}c@{\quad}c@{\qquad}l}
\toprule
$a$&$I$&$J$&$K$&classical Schubert product\\
\midrule
1 &$(2,12)$&$(6,12)$&$(5,12)$&$\sigma_{(9)}\sigma_{(5)}\supset\sigma_{(10,4)}$\\
2 &$(4,12)$&$(6,12)$&$(4,11)$&$\sigma_{(7)}\sigma_{(5)}\supset\sigma_{(9,3)}$\\
3 &$(4,12)$&$(8,12)$&$(3,10)$&$\sigma_{(7)}\sigma_{(3)}\supset\sigma_{(8,2)}$\\
4 &$(4,12)$&$(10,12)$&$(2,9)$&$\sigma_{(7)}\sigma_{(1)}\supset\sigma_{(7,1)}$\\
5 &$(5,10)$&$(8,12)$&$(5,9)$&$\sigma_{(6,2)}\sigma_{(3)}\supset\sigma_{(7,4)}$\\
6 &$(6,9)$&$(6,12)$&$(6,10)$&$\sigma_{(5,3)}\sigma_{(5)}\supset\sigma_{(8,5)}$\\
7 &$(6,9)$&$(8,12)$&$(5,9)$&$\sigma_{(5,3)}\sigma_{(3)}\supset\sigma_{(7,4)}$\\
8 &$(6,12)$&$(8,12)$&$(4,7)$&$\sigma_{(5)}\sigma_{(3)}\supset\sigma_{(5,3)}$\\
9 &$(7,11)$&$(8,12)$&$(4,7)$&$\sigma_{(4,1)}\sigma_{(3)}\supset\sigma_{(5,3)}$\\
10&$(8,12)$&$(8,12)$&$(3,6)$&$\sigma_{(3)}\sigma_{(3)}\supset\sigma_{(4,2)}$\\
\bottomrule
\end{tabular}
\end{table}

\begin{table}[ht]
\centering
\small
\caption{A violated inequality for each rank-one survivor.}
\label{tab:compactcoverage}
\begin{tabular}{c@{\qquad}c@{\qquad}c@{\qquad}c}
\toprule
$\nu$&$(m_0,m_1,m_2)$&$a$&gap\\
\midrule
6 &(3,6,3)&6&$2/117$\\
6 &(4,4,4)&2&$1/234$\\
6 &(5,2,5)&1&$1/234$\\
6 &(5,5,2)&2&$1/234$\\
6 &(6,3,3)&1&$1/234$\\
6 &(7,4,1)&2&$1/18$\\
8 &(1,7,4)&9&$1/18$\\
8 &(2,5,5)&5&$2/117$\\
8 &(3,3,6)&10&$5/117$\\
8 &(3,6,3)&7&$2/117$\\
8 &(4,1,7)&8&$5/117$\\
8 &(4,4,4)&10&$5/117$\\
8 &(5,2,5)&8&$5/117$\\
8 &(5,5,2)&3&$1/234$\\
8 &(6,3,3)&8&$5/117$\\
8 &(7,4,1)&3&$1/18$\\
10&(4,4,4)&4&$1/234$\\
\bottomrule
\end{tabular}
\end{table}

It follows that none of the 217 orbit representatives admits a unitary
factorization.  Lemma~\ref{lem:branchsymmetry} then excludes all 1302 original
tuples, so no unitary factorization~\eqref{eq:triple} exists.

\Needspace{9\baselineskip}
To illustrate the normalization and indexing conventions, expand row~6 for
the first entry of Table~\ref{tab:compactcoverage}.  In the zero-th $\lambda$
branch, take the
$(-1)$-multiplicity of $Q$ to be $6$, and take the
$(1,\omega,\omega^2)$-multiplicities of $P$ to be $(3,6,3)$.  The relevant
alcove coordinates satisfy
\[
 u_6+u_9=\frac16,\qquad v_6+v_{12}=0,\qquad
 w_6+w_{10}=-\frac{35}{234}.
\]
The Pieri product
\[
 \sigma_{(5,3)}\star\sigma_{(5)}\supset\sigma_{(8,5)}
\]
corresponds in our convention to
$(I,J,K,d)=((6,9),(6,12),(6,10),0)$.  Thus the left side of
Equation~\eqref{eq:AW} is
\[
 \frac16-\frac{35}{234}=\frac{2}{117}>0=d.
\]
\end{proof}

\begin{proof}[Proof of Theorem~\ref{thm:main}]
If a unitary modular tensor category realized the candidate datum, its ribbon
structure would supply the unitary $B_3$ representation on
$V_{C_{[10]}}$.  Proposition~\ref{prop:Rsigns} and the fusion rules force its
generator spectrum to be~\eqref{eq:C10spectrum}, while the ribbon full twist
forces the central character used in Lemma~\ref{lem:central}.  This contradicts
Theorem~\ref{thm:horn}.
\end{proof}

\section{The remaining non-unitary question}\label{sec:nonunitary}

The proof of Theorem~\ref{thm:main} uses a positive-definite Hermitian
structure twice: categorical unitarity makes the braid representation
unitary, and the Agnihotri--Woodward theorem describes products of unitary
conjugacy classes.  In a non-unitarizable ribbon category, $P^3=Q^2=I$ still
holds and each matrix is diagonalizable over $\C$, but $P$ and $Q$ need not be
unitary for one common positive-definite form.  Consequently the Horn
inequalities need not hold.

Two possible routes to absolute non-categorifiability remain:
\begin{enumerate}[label=(N\arabic*)]
 \item prove that no complex matrices with the prescribed semisimple
 conjugacy classes satisfy $D=P^{-1}Q$, $P^3=Q^2=I$; or
 \item eliminate all solutions of the pentagon and hexagon equations for the
 based ring without using positivity.
\end{enumerate}
Neither step is claimed here.

\section{Conclusion}

The candidate $12^{39,406}_{4,1276.}$ illustrates a sharp limitation of
genus-one consistency tests.  Its modular relations and Verlinde coefficients
pass exact checks.  Nonetheless, the twelve-dimensional three-strand
fusion-path representation for $t=C_{[10]}$ is impossible in a unitary
category.  The multiplicative Horn inequalities therefore give a finite,
exact obstruction to unitary categorification.  The non-unitary problem
remains open.

\appendix

\section{Exact certificates and reproducibility}\label{sec:certificates}

The accompanying code and source data are supplied as arXiv ancillary
files in the \texttt{anc/} directory of the complete source package.
The file \texttt{anc/README.md} records the dependencies, commands, and
expected outputs.  After extracting the complete package, change into
\texttt{anc/}; all paths and commands below are relative to that directory.
Keep the \texttt{certificates/} and \texttt{sources/} subdirectories together:
the scripts locate their inputs relative to their own file locations.
In particular, the original rank-twelve GAP table is retained at
\texttt{sources/2308.09670v2/anc/NsdGOL12.g}; only Galois orbit 59 is read.

The exact computations accompanying this work are separated according to
their logical roles.
\begin{enumerate}
 \item \texttt{check\_rank12\_B3\_spectra.wls} derives the fusion-path
 dimensions, the $R_{BB}$ exponents, and the determinant constraints.
 \item \texttt{check\_rank12\_B3\_quantum\_horn.py} enumerates the Horn
 inequalities using \texttt{lrcalc} version 2.1 \cite{lrcalc} and performs
 every comparison with Python \texttt{Fraction} objects.  The compact tables
 above are obtained by restricting that exact output to $b=0$ and retaining
 the ten displayed witnesses.
\end{enumerate}
The decisive command is
\begin{verbatim}
python certificates/check_rank12_B3_quantum_horn.py \
  --sector C10 --ranks 1,2
\end{verbatim}
and its required final summary is \texttt{survivor\_count: 0}.  A separate
floating-point optimization probe is diagnostic only and is not used in any
proof.

For publication-grade archival, the aggregate Horn output should be
accompanied by a machine-readable witness for every rejected branch and a
small independent verifier that checks the recorded quantum-product
coefficient and rational inequality.  The mathematical theorem uses only the
finite exact inequalities; no conclusion depends on a numerical tolerance.

\section*{Acknowledgments}

We thank Professor Zhengwei Liu for his guidance and valuable discussions.
The ideas underlying this article originated from GPT-5.6-Sol, which was also
used to assist with drafting and editing the manuscript. The authors
subsequently reviewed and verified the manuscript and take full responsibility
for its mathematical arguments and conclusions.

\end{document}